\documentclass[12pt,a4paper]{article}

\usepackage{epsf,epsfig,amsfonts,amsgen,amsmath,amstext,amsbsy,amsopn,amsthm
}
\usepackage{amsmath,times,mathptmx}
\usepackage{amsfonts,amsthm,amssymb}
\usepackage{amsfonts}
\usepackage{graphics}
\usepackage{latexsym,bm}
\usepackage{amsfonts,amsthm,amssymb,bbding}
\usepackage{indentfirst}
\usepackage{graphicx}
\usepackage{color}
\usepackage[colorlinks=true,anchorcolor=blue,filecolor=blue,linkcolor=blue,urlcolor=blue,citecolor=blue]{hyperref}
\usepackage{float}
\usepackage{tikz}
\evensidemargin=\oddsidemargin\topmargin=-1.5cm

\newtheorem{thm}{Theorem}[section]

\newtheorem{ques}{Question}[section]
\newtheorem{lem}{Lemma}[section]
\newtheorem{cor}{Corollary}[section]

\newtheorem{conj}{Conjecture}[section]
\newtheorem{claim}{Claim}[section]

\addtocounter{section}{0}

\begin{document}
\title{Spectral extremal results on trees\footnote{Supported by the National Natural Science Foundation of China (Nos.\,12271162,\,12326372), and Natural Science Foundation of Shanghai (Nos. 22ZR1416300 and 23JC1401500) and The Program for Professor of Special Appointment (Eastern Scholar) at Shanghai Institutions of Higher Learning (No. TP2022031).}}
\author{ {\bf Longfei Fang$^{a,b}$},
{\bf Huiqiu Lin$^{a}$}\thanks{Corresponding author: huiqiulin@126.com(H. Lin)}, {\bf Jinlong Shu$^{c}$}, {\bf Zhiyuan Zhang$^a$}
\\
\small $^{a}$ School of Mathematics, East China University of Science and Technology, Shanghai 200237, China\\
\small $^{b}$ School of Mathematics and Finance, Chuzhou University, Chuzhou, Anhui 239012, China\\
\small $^{c}$ School of Finance Business, Shanghai Normal University, Shanghai 200234, China\\
}

\date{}
\maketitle
{\flushleft\large\bf Abstract.}
Let ${\rm spex}(n,F)$ be the  maximum spectral radius over all $F$-free graphs of order $n$,
and ${\rm SPEX}(n,F)$ be the family of $F$-free graphs of order $n$ with spectral radius equal to ${\rm spex}(n,F)$.
Given integers $n,k,p$ with $n>k>0$ and $0\leq p\leq \lfloor(n-k)/2\rfloor$,
let $S_{n,k}^{p}$ be the graph obtained from $K_k\nabla(n-k)K_1$ by embedding $p$ independent edges within its independent set, where `$\nabla$' means the join product.
For $n\geq\ell\geq 4$, let $G_{n,\ell}=S_{n,(\ell-2)/2}^{0}$ if $\ell$ is even,
  and  $G_{n,\ell}=S_{n,(\ell-3)/2}^{1}$ if $\ell$ is odd.
Cioab\u{a}, Desai and Tait [SIAM J. Discrete Math. 37 (3) (2023) 2228--2239] showed that
for $\ell\geq 6$ and sufficiently large $n$,
if $\rho(G)\geq \rho(G_{n,\ell})$, then $G$ contains all trees of order $\ell$ unless $G=G_{n,\ell}$.
They further posed a problem to study ${\rm spex}(n,F)$ for various specific trees $F$.
Fix a tree $F$ of order $\ell\geq 6$,
let $A$ and $B$ be two partite sets of $F$ with $|A|\leq |B|$, and set $q=|A|-1$.
We first show that any graph in ${\rm SPEX}(n,F)$ contains a spanning subgraph $K_{q,n-q}$ for $q\geq 1$ and  sufficiently large $n$.
Consequently, $\rho(K_{q,n-q})\leq {\rm spex}(n,F)\leq \rho(G_{n,\ell})$,  we further respectively characterize all trees $F$ with these two equalities holding.
Secondly, we characterize the spectral extremal graphs for some specific trees and provide asymptotic spectral extremal values of the remaining trees.
In particular, we characterize the spectral extremal graphs for all spiders, surprisingly, the extremal graphs are not always the spanning subgraph of $G_{n,\ell}$.

\begin{flushleft}
\end{flushleft}

\section{Introduction}

Given a graph $G$, let $A(G)$ be its  adjacency matrix, and $\rho(G)$ or $\rho(A(G))$ be its spectral radius (i.e.,
the largest eigenvalue  of $A(G)$).
Given a graph family $\mathcal{F}$, a graph is said to be \emph{$\mathcal{F}$-free}
if it does not contain any copy of $F\in \mathcal{F}$.
For convenience, we write $F$-free instead of $\mathcal{F}$-free if $\mathcal{F}=\{F\}$.
In 2010, Nikiforov \cite{Nikiforov2010} proposed the following Brualdi-Soheid-Tur\'{a}n type problem:
What is the maximum spectral radius in any $F$-free graph of order $n$?
The aforementioned value is called the \emph{spectral extremal value} of $F$ and denoted by ${\rm spex}(n,F)$.
An $F$-free graph $G$ is said to be \textit{extremal} for ${\rm spex}(n,F)$,
   if $|V(G)|=n$ and $\rho(G)={\rm spex}(n,F)$.
Denote by ${\rm SPEX}(n,F)$ the family of extremal graphs for  ${\rm spex}(n,F)$.
In the past decades, the Brualdi-Soheid-Tur\'{a}n type problem has been studied by many researchers  for many specific graphs,
such as complete graphs \cite{Nikiforov2007,Wilf1986}, odd cycles \cite{Nikiforov2008}, even cycles  \cite{CDT2022+,Nikiforov2007,ZB2012,ZL2020}, paths \cite{Nikiforov2010} and wheels \cite{Cioaba2022,ZHL2021}.
For more information, we refer the reader to \cite{DKL2022,FZL2023+,LP2022,L-P2022,LNW2021,LN2021,TAIT2019,TAIT2017,WANG2023}.

Fix a tree $F$ of order $\ell\geq 4$,
let $A$ and $B$ be two partite sets of $F$ with $|A|\leq |B|$, and set $q=|A|-1$.
If $q=0$, then we can see that $F$ is a star, and the spectral extremal result is trivial.
It remains the case $q\geq 1$.
Obviously, $K_{q,n-q}$ is $F$-free.
Then it is natural to consider the following result, which will be frequently used in the following.

\begin{thm}\label{theorem1.1}
For $q\geq 1$ and sufficiently large $n$, any graph in ${\rm SPEX}(n,F)$ contains a spanning subgraph $K_{q,n-q}$.
\end{thm}

Given integers $n,k,p$ with $n>k>0$ and $p\in \{0,\dots, \lfloor(n-k)/2\rfloor\}$,
let $S_{n,k}^{p}$ be the graph obtained from $K_k\nabla(n-k)K_1$ by embedding $p$ independent edges into $(n-k)K_1$, where `$\nabla$' means the join product.
For $n\geq\ell\geq 4$, set $G_{n,\ell}=S_{n,(\ell-2)/2}^{0}$ if $\ell$ is even and $G_{n,\ell}=S_{n,(\ell-3)/2}^{1}$ otherwise.
Nikiforov \cite{Nikiforov2010} posed the following conjecture,
which is a spectral version of the well-known Erd\H{o}s-S\'{o}s Conjecture that any graph of average degree larger than $\ell-2$ contains all trees of order $\ell$.

\begin{conj}\label{conj2.1}\emph{(\cite{Nikiforov2010})}
Let $\ell\geq 6$ and $G$ be a graph of sufficiently large order $n$.
If $\rho(G)\geq \rho(G_{n,\ell})$, then $G$ contains all trees of order $\ell$ unless $G=G_{n,\ell}$.
\end{conj}

The validity of Conjecture \ref{conj2.1} for $P_{\ell}$ was proved by Nikiforov \cite{Nikiforov2010},
for all brooms was proved by Liu, Broersma and Wang \cite{LHW2022},
for the family of all $\ell$-vertex trees with diameter at most 4 was proved by
Hou, Liu, Wang, Gao and Lv \cite{HLW2021} when $\ell$ is even and
Liu, Broersma and Wang \cite{LHW2023} when $\ell$ is odd.
Very recently, Cioab\u{a}, Desai and Tait \cite{CDT2023} completely solved Conjecture \ref{conj2.1}.
Thus, Conjecture \ref{conj2.1} for the family of all $\ell$-vertex trees with given diameter is true.
Now we give a slightly stronger result.

\begin{thm}\label{cor4.1}
Let $\ell\geq 6$ and $d\in \{4,\dots, \ell-1\}$, and let $G$ be a graph of sufficiently large order $n$.\\
(i) If at least one of $\ell$ and $d$ is even, then there exists a tree $F$ of order $\ell$ and diameter $d$ such that ${\rm SPEX}(n,F)=\{G_{n,\ell}\}$.\\
(ii) If both $\ell$ and $d$ are odd and $\rho(G)\geq \rho(S_{n,(\ell-3)/2}^{0})$, then $G$ contains all trees of order $\ell$ and diameter $d$ unless $G=S_{n,(\ell-3)/2}^{0}$.
\end{thm}

It is interesting to find all those trees $F$ satisfying ${\rm SPEX}(n,F)=\{G_{n,\ell}\}$.

\begin{ques}\label{ques1.1}
For sufficiently large $n$,
which tree $F$ of order $\ell\geq 6$ can satisfy ${\rm SPEX}(n,F)=\{G_{n,\ell}\}$?
\end{ques}

A \emph{covering} of a graph is a set of vertices which meets all edges of the graph.
Let $\beta(G)$ denote the minimum number of vertices in a covering of $G$.
Set $\delta:=\min\{d_F(x):x\in A\}$.
Inspired by the work of Cioab\u{a}, Desai and Tait, we provide an answer to Question \ref{ques1.1}.

\begin{thm}\label{theorem1.2}
Let $n$ be sufficiently large, and $F$ be a tree of order $\ell\geq 4$.\\
(i) For even $\ell$, ${\rm SPEX}(n,F)=\{S_{n,(\ell-2)/2}^{0}\}$ if and only if $\beta(F)=\ell/2$.\\
(ii) For odd $\ell$, ${\rm SPEX}(n,F)=\{S_{n,(\ell-3)/2}^{1}\}$ if and only if $\beta(F)=(\ell-1)/2$ and $\delta\geq 2$.
\end{thm}

In \cite{CDT2023}, Cioab\u{a}, Desai and Tait also proposed the following question.

\begin{ques}\label{ques1.2}\emph{(\cite{CDT2023})}
For sufficiently large $n$,
what is the exact value of ${\rm spex}(n,F)$ for a tree $F$ of order $\ell\geq 6$?
\end{ques}

Now we give partial answers to Question \ref{ques1.2} in Theorems \ref{theorem1.3} and \ref{theorem1.4}.

\begin{thm}\label{theorem1.3}
If $q\geq 1$ and $\delta\geq 2$, then $S_{n,q}^1$ is $F$-free.
Moreover, for sufficiently large $n$,
$$\frac{q-1}{2}+\sqrt{qn-\frac{3q^2+2q+1}{4}}<{\rm spex}(n,F)\leq \rho(J)=\sqrt{qn}+\frac{q+\delta-2}{2}+O(\frac{1}{\sqrt{n}}),$$
where $J=\begin{bmatrix}
      q-1 & n-q \\
    q & \delta-1 \\
\end{bmatrix}
$.
\end{thm}

Obviously, $\beta(F)\leq |A|=q+1$.
If $\beta(F)=q+1$,
then let $\mathcal{A}=\{K_{q+1}\}$ and otherwise,
   $$\mathcal{A}=\{F[S]~|~S~\text{is a covering of}~F~\text{with}~|S|\leq q\}.$$
Denote by ${\rm ex}(n,\mathcal{A})$ the maximum size in any $\mathcal{A}$-free graph of order $n$,
and ${\rm EX}(n,\mathcal{A})$ the family of $n$-vertex $\mathcal{A}$-free graphs with ${\rm ex}(n,\mathcal{A})$ edges.
Now we give the characterization of the spectral extremal graphs for ${\rm spex}(n,F)$ when $\delta=1$.

\begin{thm}\label{theorem1.4}
For $q\geq 1$ and sufficiently large $n$, ${\rm SPEX}(n,F)\subseteq \mathcal{H}(n,q,\mathcal{A})$ if and only if $\delta=1$,
where $\mathcal{H}(n,q,\mathcal{A})=\{Q_{\mathcal{A}}\nabla (n-q)K_1~|~Q_{\mathcal{A}}\in {\rm EX}(q,\mathcal{A})\}$.
Furthermore,\\
(i) ${\rm SPEX}(n,F)=\{K_{q,n-q}\}$ if and only if $\delta=1$ and ${\rm EX}(q,\mathcal{A})\cong \{qK_1\}$;\\
(ii) ${\rm SPEX}(n,F)=\{S_{n,q}^{0}\}$ if and only if $\delta=1$ and $\beta(F)=q+1$.
\end{thm}

Particularly, we shall show that ${\rm SPEX}(n,S_{a+1,b+1})=\{K_{a,n-a}\}$,
where the double star $S_{a+1,b+1}$ is obtained from $K_{1,a}$ and $K_{1,b}$ by joining the centers with a new edge for $a\leq b$.
If $F=S_{a+1,b+1}$, then $q=a$, $\delta=1$ and $\beta(F)=2$.
By the definition of $\mathcal{A}$, we can see that ${\rm EX}(k,\mathcal{A})= \{aK_1\}$.
By Theorem \ref{theorem1.4} (i), ${\rm SPEX}(n,F)=\{K_{a,n-a}\}$ for sufficiently large $n$.

However, it seems difficult to determine ${\rm SPEX}(n,F)$ when $\delta\geq 2$, and so we leave this as a problem.
In the following, we provide asymptotic spectral extremal values of all trees.
Note that $\rho(K_{q,n-q})=\sqrt{q(n-q)}$.
From \cite{Nikiforov2010} we know $\rho(S_{n,q}^0)=\frac{q-1}{2}+\sqrt{qn-\frac{3q^2+2q+1}{4}}$.
Combining these with Theorems \ref{theorem1.3} and \ref{theorem1.4}, we have
\begin{eqnarray}\label{alg001}
{\rm spex}(n,F)=\sqrt{qn}+O(1).
\end{eqnarray}

A tree of order $\ell\geq 4$ is said to be a \emph{spider} if it contains at most one vertex of degree at least 3.
The vertex of degree at least 3 is called the center of the spider (if any vertex is of degree 1 or 2, then the spider is a path and any vertex of degree two can be taken to be the center).
A leg of a spider is a path from the center to a leaf, and the length of a leg is the number of its edges.
Let $k\geq 2$ and let $F$ be a spider of order $2k+3$ with $r$ legs of odd length and $s$ legs of length 1.
If $r\geq 3$ and $s\geq 1$, then $q=|A|-1=\frac12((2k+3)-(r+s)-1)\leq k-1$.
By \eqref{alg001}, we get
$${\rm spex}(n,F)=\sqrt{qn}+O(1)<\sqrt{kn}+O(1)=\rho(S_{n,k}^0)$$
 for sufficiently large $n$.
This means that every graph $G$ of order $n$ with $\rho(G)\geq \rho(S_{n,k}^0)$ contains $F$ as a subgraph.
Then we can derive the following result on spiders,
which was originally proved by Liu, Broersma and Wang \cite{LHW2022}.

\begin{cor}\label{theorem1.6}\emph{(\cite{LHW2022})}\label{cor1.1}
Let $k\geq 2$ and let $F$ be a spider of order $2k+3$ with $r$ legs of odd length and $s$ legs of length 1.
If $r\geq 3$, $2s-r\geq 2$ and $n$ is sufficiently large, then every graph $G$ of order $n$ with $\rho(G)\geq \rho(S_{n,k}^0)$ contains $F$ as a subgraph.
\end{cor}

The Erd\H{o}s-S\'{o}s Conjecture has been confirmed for some special families of spiders (see \cite{Fan2013,Fan2016,Fan2007,Wozniak1996}).
Recently, Fan, Hong and Liu \cite{Fan2018+} has resolved this conjecture for all spiders.
The spectral Erd\H{o}s-S\'{o}s Conjecture has also been confirmed for several classes of spiders (see \cite{LHW2022}).
In this paper, we completely characterize ${\rm SPEX}(n,F)$ for all spiders $F$ with $q\geq 1$.

\begin{thm}\label{theorem1.5}
Let $r_1,r_2,r_3,r,s$ and $\ell$ be non-negative integers with $r=r_1+r_2+r_3$ and $\ell\geq 4$,
and let $F$ be a spider of order $\ell$ with $r_1$ legs of odd length at least 5, $r_2$ legs of length 3, $r_3$  legs of length 1 and $s$ legs of even length.
Let $n$ be sufficiently large. Then
 $${\rm SPEX}(n,F)=\left\{
                                       \begin{array}{ll}
                                        \{S_{n,(\ell-r-1)/2}^{0}\}  & \hbox{if $s\geq 1$ and $r\geq 1$,} \\
                                        \{S_{n,(\ell-3)/2}^{1}\}  & \hbox{if $s\geq 1$ and $r=0$,}\\
                                         \{S_{n,(\ell-r-1)/2}^{1}\}  & \hbox{if $s=0$ and $r_1\geq 1$,} \\
                                        \{S_{n,(\ell-r-1)/2}^{r-1}\}  & \hbox{if  $s=0$, $r_1=0$, $r_2\geq 1$ and $r_3\in \{0,1\}$,}\\
                                        \{S_{n,(\ell-r-1)/2}^{\lfloor(2n-\ell+r+1)/{4}\rfloor}\}  & \hbox{if  $s=0$, $r_1=0$, $r_2\geq 1$ and $r_3\geq 2$.}
                                       \end{array}
                                     \right.
$$

\end{thm}

%
%

\section{Proof of Theorem \ref{theorem1.1}}\label{section2}

Before beginning our proof, we first give some  notations  not defined previously.
Let $G$ be a simple graph. We use $V(G)$ to denote
the vertex set, $E(G)$ the edge set, $|V(G)|$ the number of vertices, $e(G)$ the number of edges,
$\nu(G)$ the maximum number of independent edges, respectively.
Given a vertex $v\in V(G)$ and  a vertex subset $S\subseteq V(G)$,
we denote  $N_G(v)$ the set of neighbors of $v$ in $G$,
and let $N_S(v)=N_G(v)\cap S$, $d_S(v)=|N_S(v)|$.
Let $G[S]$ (resp. $G-S$) be the subgraph of $G$ induced by $S$ (resp. $V(G)\setminus S$).
Since $|A|+|B|=\ell$ and $|A|\leq |B|$, we have $|A|\leq \ell/2$.
Moreover,
\begin{eqnarray}\label{alg002}
 \beta(F)\leq |A|\leq \ell/2.
\end{eqnarray}
By the definition of $q$, we obtain
\begin{eqnarray}\label{alg003}
F\not\subseteq K_{q,n-q}~\text{and}~F\subseteq K_{q+1,\ell}.
\end{eqnarray}
A standard graph theory exercise shows that
for any tree $F$ with $\ell\geq 2$ vertices,
\begin{eqnarray}\label{alg004}
\frac12(\ell-2) n\leq {\rm ex}(n,F)\leq (\ell-2) n.
\end{eqnarray}

In this section, we always assume that $n$ is sufficiently large and $G^{\star}$ is an extremal graph to ${\rm spex}_{\mathcal{P}}(n,F)$, and let $\rho^{\star}$ denote its spectral radius.
By Perron-Frobenius theorem, there exists a non-negative eigenvector $X=(x_1,\ldots,x_n)^\mathrm{T}$ corresponding to $\rho^{\star}$.
Choose a vertex $u^{\star}\in V(G^{\star})$ with $x_{u^{\star}}=\max\{x_i~|~i=1,2,\dots,n\}=1$.
We also choose a positive constant $\varepsilon$ and a positive integer $\phi$ satisfying
\begin{eqnarray}\label{alg005}
 \varepsilon<\frac{1}{(2\ell )^4}~~\text{and}~~
 \frac{3}{(2\ell)^{\phi-1}}<\min\left\{q\varepsilon, \frac{\varepsilon}{4\ell}\right\},
\end{eqnarray}
which will be frequently used later. First, we give a rough estimation on $\rho^{\star}$.

\begin{lem}\label{lemma2.1.}
For sufficiently large $n$, we have $\sqrt{q(n-q)}\leq \rho^{\star}\leq \sqrt{2\ell n}.$
\end{lem}

\begin{proof}
By \eqref{alg003}, $K_{q,n-q}$ is $F$-free.
Hence, $\rho^{\star}\geq \rho(K_{q,n-q})=\sqrt{q(n-q)}$ as $G^{\star}$ is an extremal graph to ${\rm spex}_{\mathcal{P}}(n,F)$.
Given a vertex $u\in V(G^{\star})$, denote by $N_i(u)$ the set of vertices
at distance $i$ from $u$.
Now we prove the upper bound.
Note that
$$(\rho^{\star})^2 x_{u^{\star}}=\sum_{u\in N_1(u^{\star})}\sum_{w\in N_1(u)}x_w
\leq |N_1(u^{\star})|+2e\big(N_1(u^{\star})\big)+e\big(N_1(u^{\star}),N_2(u^{\star})\big)
\leq 2e(G^{\star}).$$
Combining \eqref{alg004} gives $(\rho^{\star})^2\leq 2e(G^{\star})\leq 2\ell n$, which leads to $\rho^{\star}\leq \sqrt{2\ell n}$, as desired.
\end{proof}

Set $L^{\eta}=\{u\in V(G^{\star})~|~x_u\ge (2\ell)^{-\eta}\}$ for some positive integer $\eta$.
We shall constantly give an upper bound of $|L^{\eta}|$ and  a lower bound for
degrees of vertices in $L^{\eta}$ (see Lemmas \ref{lemma2.2.}--\ref{lemma2.5.}).

\begin{lem}\label{lemma2.2.}
For every positive integer $\mu$, we have $|L^{\mu}|\leq (2\ell)^{\mu+2}$.
\end{lem}

\begin{proof}
By Lemma \ref{lemma2.1.}, we get
$$\frac{\sqrt{q(n-q)}}{(2\ell)^{\eta}}\leq\rho^{\star} x_u=\sum_{v\in N_{G^{\star}}(u)}x_v\le d_{G^{\star}}(u)$$
for each $u\in L^{\eta}$.
Summing this inequality over all vertices in $L^{\eta}$, we obtain
\begin{align*}
 |L^{\eta}|\sqrt{q(n-q)}\cdot\frac{1}{(2\ell)^{\eta}}
\leq \sum_{u\in V(G^{\star})}d_{G^{\star}}(u)
\leq 2{\rm ex}(n,F)\leq  2(\ell-2) n.
\end{align*}
Consequently,
$|L^{\eta}|\leq n^{0.6}$ for sufficiently large $n$.

Given an arbitrary vertex $u\in V(G^{\star})$.
For simplicity,
we use $N_i$, $L_i^{\eta}$ and $\overline{L_i^{\eta}}$ instead of $N_i(u)$, $N_i(u)\cap L^{\eta}$ and $N_i(u)\setminus L^{\eta}$, respectively.
By Lemma \ref{lemma2.1.}, we have
\begin{eqnarray}\label{alg006}
q(n-q)x_{u}\leq (\rho^{\star})^2x_{u}=d_{G^{\star}}(u)x_u+\sum_{v\in N_1}\sum_{w\in N_1(v)\setminus\{u\}}x_w.
\end{eqnarray}

Since $N_1(v)\setminus\{u\}\subseteq N_1\cup N_2$,
we get $(N_1(v)\setminus\{u\})\cap L^{\eta}\subseteq L_1^{\eta}\cup L_2^{\eta}$ and
$(N_1(v)\setminus\{u\})\cap \overline{L^{\eta}}\subseteq \overline{L_1^{\eta}}\cup \overline{L_2^{\eta}}$.
Now we divide $\sum_{v\in N_1}\sum_{w\in N_1(v)\setminus\{u\}}x_w$
into two cases $w\in L_1^{\eta}\cup L_2^{\eta}$ or $w\in \overline{L_1^{\eta}}\cup \overline{L_2^{\eta}}$.
Clearly, $N_1=L_1^{\eta}\cup\overline{L_1^{\eta}}$.
In the case $w\in L_1^{\eta}\cup L_2^{\eta}$,
\begin{eqnarray}\label{alg007}
\sum_{v\in N_1}\sum_{w\in (L_1^{\eta}\cup L_2^{\eta})}\!\!x_w
\leq
\big(2e(L_1^{\eta})+e(L_1^{\eta},L_2^{\eta})\big)+\sum_{v\in\overline{L_1^{\eta}}}\sum_{w\in (L_1^{\eta}\cup L_2^{\eta})}\!\!\!x_w.
\end{eqnarray}
Since $|L^{\eta}|\leq n^{0.6}$,
we have
\begin{eqnarray}\label{alg008}
2e(L_1^{\eta})+e(L_1^{\eta},L_2^{\eta})\leq
2e(L^{\eta})
\leq 2\ell|L^{\eta}|\leq 2\ell n^{0.6}.
\end{eqnarray}

Now we deal with the case $w\in \overline{L_1^{\eta}}\cup \overline{L_2^{\eta}}$.
Recall that $x_w\leq \frac{1}{(2\ell)^{\eta}}$ for $w\in\overline{L_1^{\eta}}\cup\overline{L_2^{\eta}}$.
Then
\begin{eqnarray}\label{alg009}
\sum_{v\in N_1}\sum_{w\in\overline{L_1^{\eta}}\cup\overline{L_2^{\eta}}}\!\!\!x_w \le
\Big(e(L_1^{\eta},\overline{L_1^{\eta}}\cup\overline{L_2^{\eta}})
+2e(\overline{L_1^{\eta}})+e(\overline{L_1^{\eta}},\overline{L_2^{\eta}})\Big)\frac{1}{(2\ell)^{\eta}}
\leq \frac{n}{(2\ell)^{\eta-1}},
\end{eqnarray}
where $e(L_1^{\eta},\overline{L_1^{\eta}}\cup\overline{L_2^{\eta}})
+2e(\overline{L_1^{\eta}})+e(\overline{L_1^{\eta}},\overline{L_2^{\eta}})\le 2e(G^{\star})\leq 2{\rm ex}(n,F)\leq 2\ell n$ by \eqref{alg004}.

Combining (\ref{alg006})-(\ref{alg009}), we obtain
\begin{eqnarray}\label{alg010}
 qn x_{u}&<& q^2x_{u}+d_{G^{\star}}(u)+ 2\ell n^{0.6}+\sum_{v\in\overline{L_1^{\eta}}}\sum_{w\in (L_1^{\eta}\cup L_2^{\eta})}\!\!\!x_w+\frac{n}{(2\ell)^{\eta-1}}\nonumber\\
 &<&d_{G^{\star}}(u)+\sum_{v\in\overline{L_1^{\eta}}}\sum_{w\in (L_1^{\eta}\cup L_2^{\eta})}\!\!\!x_w+\frac{2n}{(2\ell)^{\eta-1}}.
\end{eqnarray}

Now we show that $d_{G^{\star}}(u)\geq\frac{n}{(2\ell)^{\mu+1}}$ for any $u\in L^{\mu}$.
By \eqref{alg004}, we have
\begin{eqnarray}\label{alg011}
e\big(\overline{L_1^{\eta}},L_1^{\eta}\cup L_2^{\eta}\big)
   \leq \ell \big(|\overline{L_1^{\eta}}|+|L_1^{\eta}\cup L_2^{\eta}|\big)\leq \ell d_{G^{\star}}(u)+\frac{n}{(2\ell)^{\eta-1}},
\end{eqnarray}
where the last inequality holds as $|\overline{L_1^{\eta}}|\leq d_{G^{\star}}(u)$, $|L^{\eta}|\leq n^{0.6}$ and $n$ is sufficiently large.
Combining (\ref{alg010}) and (\ref{alg011}), we obtain
$qn x_{u}<(\ell+1)d_{G^{\star}}(u)+\frac{3n}{(2\ell)^{\eta-1}}$.
Clearly, $x_{u}\geq \frac{1}{(2\ell)^{\mu}}$ as $u\in L^{\mu}$.
Combining these with $\eta=\mu+2$ we obtain
\begin{eqnarray*}
\frac{(\ell+4)n}{(2\ell)^{\mu+1}}\leq \frac{q n}{(2\ell)^{\mu}}\leq (\ell+1)d_{G^{\star}}(u)+\frac{3n}{(2\ell)^{\mu+1}},
\end{eqnarray*}
where the first inequality holds as $q\geq 1$ and $\ell\geq 4$.
Consequently, $d_{G^{\star}}(u)\geq \frac{n}{(2\ell)^{\mu+1}}$.
Summing this inequality over all vertices in $L^{\mu}$, we obtain
$$|L^{\mu}|\frac{n}{(2\ell)^{\mu+1}}\leq \sum_{u\in L^\mu}d_{G^{\star}}(u)\leq 2e(G^{\star})\leq 2{\rm ex}(n,F)\leq 2\ell n,$$
which leads to $|L^{\mu}|\leq (2\ell)^{\mu+2}$, completing the proof.
\end{proof}

\begin{lem}\label{lemma2.3.}
For every positive integer $\mu$ and every $u\in L^{\mu}$,
we have $d_{G^{\star}}(u)\geq \left(x_u-\varepsilon\right)n$.
\end{lem}

\begin{proof}
Let $\overline{L_1^{\eta}}'$ be the subset of $\overline{L_1^{\eta}}$
in which each vertex has at least $q$ neighbors in $L_1^{\eta}\cup L_2^{\eta}$.
We first claim that $|\overline{L_1^{\eta}}'|\leq \ell\binom{|L_1^{\eta}\cup L_2^{\eta}|}{q}$.
If $|L_1^{\eta}\cup L_2^{\eta}|\leq q-1$, then $\overline{L_1^{\eta}}'$ is empty, as desired.
Now we deal with the case $|L_1^{\eta}\cup L_2^{\eta}|\geq q$.
Suppose to the contrary that $|\overline{L_1^{\eta}}'|> \ell\binom{|L_1^{\eta}\cup L_2^{\eta}|}{q}$.
Since there are only $\binom{|L_1^{\eta}\cup L_2^{\eta}|}{q}$ options for vertices in $\overline{L_1^{\eta}}'$ to choose a set of $q$ neighbors from $L_1^{\eta}\cup L_2^{\eta}$,
we can find a set of $q$ vertices in $L_1^{\eta}\cup L_2^{\eta}$ with at least $\lfloor|\overline{L_1^{\eta}}'|/\binom{|L_1^{\eta}\cup L_2^{\eta}|}{q}\rfloor\geq \ell$ common neighbors in $\overline{L_1^{\eta}}'$.
Moreover, one can observe that $u\notin L_1^{\eta}\cup L_2^{\eta}$
and $\overline{L_1^{\eta}}'\subseteq\overline{L_1^{\eta}}\subseteq N_1(u)$.
Hence, $G^{\star}$ contains a copy of $K_{q+1,\ell}$, and so contains a copy of $F$ by \eqref{alg003}, which gives a contradiction.
The claim holds.
Thus,
\begin{eqnarray}\label{alg012}
e(\overline{L_1^{\eta}},L_1^{\eta}\cup L_2^{\eta})
\leq (q-1)|\overline{L_1^{\eta}}\setminus \overline{L_1^{\eta}}'|\time 2
    +|L_1^{\eta}\cup L_2^{\eta}||\overline{L_1^{\eta}}'|
\leq (q-1)d_{G^{\star}}(u)+\frac{n}{(2\ell)^{\eta-1}},
\end{eqnarray}
where the last inequality holds because both $|L_1^{\eta}\cup L_2^{\eta}|\leq |L^{\eta}|$ and
$|\overline{L_1^{\eta}}'|\leq \ell\binom{|L_1^{\eta}\cup L_2^{\eta}|}{q}$ are constants.
Combining  \eqref{alg010} and \eqref{alg012}, we have
$qn x_{u}\leq q d_{G^{\star}}(u)+\frac{3n}{(2\ell)^{\eta-1}}$.
Setting $\eta=\phi$, by \eqref{alg005} we get $d_{G^{\star}}(u)\geq (x_u-\varepsilon)n$.
\end{proof}

\begin{lem}\label{lemma2.4.}
For every $u\in L^1$, $x_u\geq 1-\varepsilon$
and $|N_1(u)|\geq (1-2\varepsilon)n$.
Moreover, $|L^{1}|=q$.
\end{lem}

\begin{proof}
We first show the lower bounds of $x_u$ and $|N_1(u)|$ for any $u\in L^1$.
Suppose to the contrary that there exists a vertex $u_0\in L^1$ with $x_{u_0}<1-\varepsilon$.
Since $u_0\in L^1$, we have $x_{u_0}\geq\frac{1}{2\ell }$.
By Lemma \ref{lemma2.3.}, we get
$$|N_1(u^{\star})|\geq\left(1-\varepsilon\right)n~~~\text{and}
~~~|N_1(u_0)|\geq\left(\frac{1}{{2\ell}}-\varepsilon\right)n.$$
For convenience, we set $L_i^{\eta}=N_i(u^{\star})\cap L^{\eta}$
and $\overline{L_i^{\eta}}=N_i(u^{\star})\setminus L^{\eta}$.
By Lemma \ref{lemma2.2.},
$|L^{\eta}|\leq ({2\ell})^{\eta+2}$.
Hence, $|\overline{L_1^{\eta}}|\geq|N_1(u^{\star})|-|L^{\eta}|
\geq\big(1-2\varepsilon\big)n$. Consequently, by \eqref{alg005}
\begin{eqnarray}\label{alg013}
\big|\overline{L_1^{\eta}}\cap N_1(u_0)\big|\geq\big|\overline{L_1^{\eta}}\big|
+\big|N_1(u_0)\big|-n\ge\Big(\frac{1}{{2\ell}}-3\varepsilon\Big)n>\frac{n}{4\ell}.
\end{eqnarray}

From (\ref{alg013}) we can see that $u_0$ has a neighbor in $\overline{L_1^{\eta}}$,
which is also a neighbor of $u^{\star}$.
Thus, $u_0\in N_1(u^{\star})\cup N_2(u^{\star})$.
Note that $u_0\in L^1\subseteq L^{\eta}$. Thus, $u_0\in L_1^{\eta}\cup L_2^{\eta}$.
Now, applying $u=u^{\star}$ to \eqref{alg010} gives
\begin{eqnarray*}
qn
\!&\leq&\! |N_1(u^{\star})|+\frac{2n}{(2\ell)^{\eta-1}}
+e\big(\overline{L_1^{\eta}},(L_1^{\eta}\cup L_2^{\eta})\setminus\{u_0\}\big)
+e\big(\overline{L_1^{\eta}},\{u_0\}\big)x_{u_0}\\
\!&\leq&\! |N_1(u^{\star})|+\frac{2n}{(2\ell)^{\eta-1}}
+e\big(\overline{L_1^{\eta}},(L_1^{\eta}\cup L_2^{\eta})\big)
+e\big(\overline{L_1^{\eta}},\{u_0\}\big)\big(x_{u_0}-1\big),
\end{eqnarray*}
where $x_{u_0}-1<-\varepsilon$
by the previous assumption.
Combining this with (\ref{alg012}) and setting $\eta=\phi$, we have
$$q n
  \leq q|N_1(u^{\star})|+\frac{3n}{(2\ell)^{\phi-1}}
        -\varepsilon e\big(\overline{L_1^{\eta}},\{u_0\}\big),$$
which yields that
$\big|\overline{L_1^{\eta}}\cap N_1(u_0)\big|=e\big(\overline{L_1^{\eta}},\{u_0\}\big)<\frac{n}{4\ell}$ by \eqref{alg005},
contradicting (\ref{alg013}).
Therefore, $x_u\geq 1-\varepsilon$ for each $u\in L^1$.
Furthermore,
it follows from Lemma \ref{lemma2.3.} that for each $u\in L^1$,
$|N_1(u)|\geq (1-2\varepsilon)n.$

Finally, we prove that $|L^1|=q$.
We first suppose $|L^{1}|\geq q+1$.
Note that every vertex $u\in L^{1}$
has at most $2\varepsilon n$ non-neighbors.
It follows that any $q$ vertices in $L^{1}$ have at least $n-2q \varepsilon n\geq \frac{n}{2}$ common neighbors by \eqref{alg005}.
Hence, $G^{\star}$ contains a copy of $K_{q+1,\ell}$, and so contains a copy of $F$ by \eqref{alg003}, which gives a contradiction.
Hence, $|L^1|\leq q$.

Next, suppose that $|L^1|\leq q-1$.
Since $u^{\star}\in L^{1}\setminus (L_1^{\eta}\cup L_2^{\eta})$,
we have $|(L_1^{\eta}\cup L_2^{\eta})\cap L^{1}|\leq q-2$.
We can further obtain that
    $$e\big(\overline{L_1^{\eta}},(L_1^{\eta}\cup L_2^{\eta})\cap L^1\big)\leq|\overline{L_1^{\eta}}|\cdot|(L_1^{\eta}\cup L_2^{\eta})\cap L^1| \leq (q-2)n.$$
By \eqref{alg004}, we have
$e\big(\overline{L_1^{\eta}},(L_1^{\eta}\cup L_2^{\eta})\setminus L^{\eta}\big)\leq e(G^{\star})<\ell n.$
Furthermore, by the definition of $L^1$, we know that
$x_w<\frac{1}{{2\ell}}$
for each $w\in(L_1^{\eta}\cup L_2^{\eta})\setminus L^1.$
Applying $u=u^{\star}$ and $\eta\geq 2$ to \eqref{alg010} gives
\begin{eqnarray*}
qn
\!&\leq&\!d_{G^{\star}}(u^\star)+\sum_{v\in\overline{L_1^{\eta}}}\sum_{w\in (L_1^{\eta}\cup L_2^{\eta})}x_w+\frac{2n}{(2\ell)^{\eta-1}}\\
\!&\leq&\!
 \Big(|N_1(u^{\star})|\!+\frac{2n}{(2\ell)^{\eta-1}}+\!
e\big(\overline{L_1^\eta},(L_1^\eta\cup L_2^\eta)\cap L^1\big)\Big)
\!+\!e\Big(\overline{L_1^\eta},(L_1^\eta\cup L_2^\eta)\setminus L^{1}\Big)\frac{1}{{2\ell}}\\
\!&\leq&\! \Big(n+\frac{2n}{2\ell}\Big)+(q-2)n+\ell n\cdot\frac{1}{2\ell}\\
\!&\leq&\! \left(q-\frac14\right)n ~~~~~~~~~~~~~(\text{as}~\ell\geq 4),
\end{eqnarray*}
which gives a contradiction.
Therefore, $|L^1|=q$.
\end{proof}

For convenience,
we use $L$, $L_i$ and $\overline{L_i}$ instead of $L^{1}$, $N_i(u)\cap L^{1}$ and $N_i(u)\setminus L^{1}$,
respectively.
Now, let $R_1$ be the subset of $V(G^{\star})\setminus L$
in which every vertex is a non-neighbor of some vertex in $L$
and $R=V(G^{\star})\setminus (L\cup R_1)$.
Thus, $|R_1|\leq 2\varepsilon n|L| \leq \frac{n}{(2\ell)^3}$ by \eqref{alg005},
and so $|R|=n-|L|-|R_1|\geq \frac{n}{2}$.
Now, we prove that the eigenvector entries of vertices in $R\cup R_1$ are small.

\begin{lem}\label{lemma2.5.}
Let $u\in R\cup R_1$. Then $x_u\le \frac{1}{2\ell^2}$.
\end{lem}

\begin{proof}
For any vertex $u\in R\cup R_1$, we can see that
\begin{eqnarray}\label{alg014}
d_{R}(u)\leq \ell-1.
\end{eqnarray}
Indeed, if $d_{R}(u)\geq \ell$, then $G^{\star}[N_G(u)\cup \{u\}\cup L]$ contains a copy of $K_{q+1,\ell}$,
and so contains a copy of $F$ by \eqref{alg003}, a contradiction.
By Lemma \ref{lemma2.4.} and \eqref{alg002}, $|L|=q\leq (\ell-2)/{2}$.
Then, $$d_{G^{\star}}(u)=d_L(u)+d_{R}(u)+d_{R_1}(u)\leq \frac32\ell+d_{R_1}(u).$$
Note that $|R_1|\leq \frac{n}{(2\ell )^3}$ and $e(R_1)\leq \ell|R_1|$ by \eqref{alg004}.
Thus,
  $$\rho^{\star}\sum_{u\in R_1}x_u\le \sum_{u\in R_1}d_{G^{\star}}(u)
  \leq \sum_{u\in R_1}\left(\frac{3}{2}\ell+d_{R_1}(u)\right)
  \leq \frac32\ell|R_1|+2e(R_1)\leq \frac72\ell|R_1|
  \leq \frac{7n}{16\ell^2},$$
which yields $\sum_{u\in R_1}x_u\leq \frac{7n}{16\ell^2\rho^{\star}}$.
Combining $|L|\leq (\ell-2)/{2}$ and \eqref{alg014}, we obtain
        $$\rho^{\star} x_u=\sum_{v\in N_{G^{\star}}(u)}x_v\leq  \sum_{v\in N_L(u)}x_v+\sum_{v\in N_{R}(u)}x_v+\sum_{v\in  N_{R_1}(u)}x_v\leq \frac32\ell+\frac{7n}{16\ell^2\rho^{\star}}.$$
Note that $\rho^{\star}\geq \sqrt{q(n-q)}\geq \sqrt{n-1}$.
Dividing both sides by $\rho^{\star}$, we get
\begin{eqnarray*}
x_u \!\leq\! \frac{3\ell}{2\rho^{\star}}+\frac{7n}{16\ell^2(\rho^{\star})^2}
    \!\leq\! \frac{3\ell}{2\sqrt{n-1}}+\frac{7n}{16\ell^2(n-1)}
    \leq \frac{1}{2\ell^2},
\end{eqnarray*}
where the last inequality holds as $n$ is sufficiently large, as desired.
\end{proof}

Now we complete the proof of Theorem \ref{theorem1.1}.
\begin{proof}[\textbf{Proof of Theorem~\ref{theorem1.1}}]
From \eqref{alg004} we know that $e(R_1)\leq \ell|R_1|$.
Then there exists a vertex $v_1\in R_1$ with
$d_{R_1}(v_1)\leq\frac{2e(R_1)}{|R_1|}\leq2\ell$.
We modify the graph $G^{\star}$ by  deleting all edges incident to $v_1$ and joining $v_1$ to all vertices in $L$ to obtain the graph $G^{\star\star}$.
We first claim that $G^{\star\star}$ is $F$-free.
Suppose to the contrary, then $G^{\star\star}$ contains a subgraph $F'$ isomorphic to $F$.
From the modification, we can see that $v_1\in V(F')$.
Since $|R|\geq \frac{n}{2}$, we have $|R\setminus V(F')|\geq |R|-\ell>\ell.$
Then there exists a vertex  $w_1\in R\setminus V(F')$.
Clearly, $N_{G^{\star\star}}(v_1)=L\subseteq N_{G^{\star\star}}(w_1)$.
This indicates that a copy of $F$ is already present in $G^{\star}$, which gives a contradiction.
The claim holds.

Now we claim that $\rho(G^{\star\star})>\rho^{\star}$.
By (\ref{alg014}) and Lemma \ref{lemma2.5.}, we have
\begin{eqnarray}\label{alg015}
    \sum_{w\in N_{L\cup R\cup R_1}(v_1)}x_w\leq (q-1)+\sum_{w\in N_R(v_1)}x_w+\sum_{w\in N_{R_1}(v_1)}x_w
    \leq (q-1)+3\ell\cdot\frac{1}{2\ell^2},
\end{eqnarray}
By Lemma \ref{lemma2.4.}, $\sum_{w\in L}x_w\geq q(1-\varepsilon)$.
Combining this with (\ref{alg015}) and \eqref{alg005}, we have
$$\rho(G^{\star\star})-\rho^{\star}\geq \frac{2}{X^\mathrm{T}X}x_{v_1}\left(\sum_{w\in L}x_w-\sum_{w\in N_{L\cup R\cup R_1}(v_1)}x_w\right)\geq 0.
$$
If $\rho(G^{\star\star})=\rho^{\star}$,
then $x_{v_1}=0$ and  $X$ is also a non-negative eigenvector of $G^{\star\star}$ corresponding to $\rho^{\star}$.
This implies that $\rho(G^{\star\star}) x_{v_{1}}= \sum_{w\in L}x_w\geq q(1-\varepsilon)$, and so $x_{v_1}>0$, a contradiction.
Thus, $\rho(G^{\star\star})>\rho^{\star}$, contradicting that $G^{\star}$ is extremal to ${\rm spex}(n,F)$.
Therefore, $R_1$ is empty, and thus $G^{\star}$ contains a spanning subgraph $K_{q,n-q}$,
completing the proof.
\end{proof}

\section{Proofs of Theorems \ref{cor4.1}-\ref{theorem1.5}}\label{section3}

In this section, we first record several technique lemmas that we will use.

\begin{lem}\label{lemma3.1.} \emph{(\cite{TAIT2019})}
Let $H_1$ be a graph on $n_0$ vertices with maximum degree $d$ and $H_2$ be a graph on $n-n_0$ vertices with maximum degree $d'$. $H_1$ and $H_2$ may have loops or multiple edges, where loops add 1 to the degree.
Let $H=H_1\nabla H_2$. Define
$$
J^{\star}=\begin{bmatrix}
      d & n-n_0 \\  
    n_0 & d' \\  
\end{bmatrix}.
$$
Then $\rho(H)\leq \rho(J^{\star})$.
\end{lem}

The well--known K\"{o}nig--Egerv\'{a}ry Theorem is as follows.
\begin{lem}\emph{(\cite{Konig})}\label{lemma3.2.}
For any bipartite graph $G$, we have $\beta(G)=\nu(G)$.
\end{lem}

By the proof of Theorem \ref{theorem1.1}, we can see that $G^{\star}=G^{\star}[L] \nabla G^{\star}[R]$.
We then give three lemmas to characterize $G^{\star}[L]$ and $G^{\star}[R]$,
which help us to present an approach to prove the remaining theorems.

\begin{lem}\label{lem3.3}
Let $n$ be sufficiently large and $H$ be a graph of order $q$.
Then $H\nabla (n-q)K_1$ is $F$-free if and only if $H$ is $\mathcal{A}$-free.
Furthermore, if $G^{\star}[L]\cong K_q$, then $\beta(F)\geq q+1$.
\end{lem}

\begin{proof}
Suppose first that $H$ is $\mathcal{A}$-free.
Then we show that $H\nabla (n-q)K_1$ is $F$-free.
Otherwise, embed $F$ into $H\nabla (n-q)K_1$,
where $S=V(F)\cap V(H)$.
Then $F[S]\subseteq H[S]$ and $S$ is a covering set of $F$.
By the definition of $\mathcal{A}$, $F[S]\in \mathcal{A}$, which contradicts that $H$ is $\mathcal{A}$-free.
Hence, $H\nabla (n-q)K_1$ is $F$-free.
Suppose then that $H$ is not $\mathcal{A}$-free.
By the definition of $\mathcal{A}$, there exists a covering set $S$ of $F$ such that $|S|\leq q$ and $F[S]\subseteq H$.
We can further find that $H\nabla (n-q)K_1$ contains a copy of $F$.
Therefore, $H\nabla (n-q)K_1$ is $F$-free if and only if $H$ is $\mathcal{A}$-free.

By Theorem \ref{theorem1.1}, $G^{\star}[L]\nabla (n-q)K_1\subseteq G^{\star}$.
Since $G^{\star}$ is $F$-free, so does $G^{\star}[L]\nabla (n-q)K_1$.
Thus, $G^{\star}[L]$ is $\mathcal{A}$-free.
Assume that $G^{\star}[L]\cong K_q$.
Now we prove that $\beta(F)\geq q+1$. If not, then there exists a covering set $S$ of $F$ with $|S|=\beta(F)\leq q$.
Clearly, $F[S]\subseteq K_q$ and $F[S]\in \mathcal{A}$, which contradicts that $G^{\star}[L]$ is $\mathcal{A}$-free as $G^{\star}[L]\cong K_q$.
Hence, $\beta(F)\geq q+1$.
This completes the proof.
\end{proof}

Given a non-nagative integer $p\leq {b}/{2}$,
let $K_{a,b}^{p}$ be the graph obtained from $aK_1\nabla bK_1$ by embedding $p$ independent edges into the partite set of size $b$.

\begin{lem}\label{lem3.4}
Let $n$ be sufficiently large and $\delta=1$. Then $e(G^{\star}[R])=0$ and $G^{\star}[L]\in {\rm EX}(q,\mathcal{A})$.
\end{lem}

\begin{proof}
Since $\delta=1$, there exists a vertex $v\in A$ of degree 1 in $F$.
Let $A'=A\setminus \{v\}$ and $B'=B\cup \{v\}$.
Obviously, $|A'|=q$ and $F[B']$ consists of an edge and some isolated vertices,
which implies that $F\subseteq K_{q,\ell-q}^{1}$.
If $e(G^{\star}[R])\geq 1$, then $G^{\star}$ must contain a copy of $K_{q,n-q}^{1}$, and so contains a copy of $F$, a contradiction. Thus, $e(G^{\star}[R])=0$.

By Lemma \ref{lem3.3}, $G^{\star}[L]$ is $\mathcal{A}$-free,
which implies that $e(G^{\star}[L])\leq {\rm ex}(q,\mathcal{A})$.
Now we prove that $e(G^{\star}[L])={\rm ex}(q,\mathcal{A})$.
Suppose to the contrary, then $e(G^{\star}[L])<e(Q_{\mathcal{A}})$,
 where $Q_{\mathcal{A}}\in {\rm EX}(q,\mathcal{A})$.
Clearly, $e(Q_{\mathcal{A}})\leq e(K_q)= \binom{q}{2}$. By Lemma \ref{lemma2.4.} and \eqref{alg005}, we have
\begin{align*}
\sum_{uv\in E(Q_{\mathcal{A}})}x_{u}x_{v}-\sum_{uv\in E(G^{\star}[L])}x_{u}x_{v}
  &\geq e(Q_{\mathcal{A}})(1-\varepsilon)^2-e(G^{\star}[L])\\
  &>e(Q_{\mathcal{A}})-2\varepsilon e(Q_{\mathcal{A}})-e(G^{\star}[L])\\
  &\geq 1-2\varepsilon \binom{q}{2}\\
  &>0.
\end{align*}
Consequently,
\begin{align*}
    \rho(Q_{\mathcal{A}}\nabla (n-q)K_1)-\rho(G^{\star})
    &\geq  \frac{1}{X^\mathrm{T}X}{X^\mathrm{T}(A(Q_{\mathcal{A}}\nabla (n-q)K_1)-A(G^{\star}))X}\\
    &\geq  \frac{2}{X^\mathrm{T}X}\left(\sum_{uv\in E(Q_{\mathcal{A}})}x_{u}x_{v}-\sum_{uv\in E(G^{\star}[L])}x_{u}x_{v}\right)\\
    &>0.
\end{align*}
By Lemma \ref{lem3.3}, $Q_{\mathcal{A}}\nabla (n-q)K_1$ is $F$-free.
However, this contradicts that $G^{\star}$ is extremal to ${\rm spex}(n,F)$.
Hence, $e(G^{\star}[L])={\rm ex}(q,\mathcal{A})$.
From the proof in Lemma \ref{lem3.3} we know that  $G^{\star}[L]$ is $\mathcal{A}$-free.
Therefore, $G^{\star}[L]\in {\rm EX}(q,\mathcal{A})$.
\end{proof}

\begin{lem}\label{lem3.5}
Let $n$ be sufficiently large and $\delta\geq 2$. Then $S_{n,q}^1$ is $F$-free and $e(G^{\star}[R])\geq 1$.
\end{lem}

\begin{proof}
We first prove that $S_{n,q}^1$ is $F$-free,
where $Y_1$ is the set of dominating vertices of $S_{n,q}^1$  and $Y_2=V(S_{n,q}^1)\setminus Y_1$.
Otherwise, embed $F$ into $S_{n,q}^1$.
Set $A_i=A\cap Y_i$ for each $i\in \{1,2\}$.
Then $A=A_1\cup A_2$.
Since $|A|=q+1=|Y_1|+1>|A_1|$, we have $A_2\neq\varnothing$.
In  graph $F$, let $B_1$ be the set of vertices in $Y_1$ adjacent to at least one vertex in $A_2$.
Then, $B_1\subseteq B$, and thus $A_1\subseteq Y_1\setminus B_1$ as $A_1\subseteq A$.
Obviously, $S_{n,q}^1[Y_2]$ contains exactly one edge, say $e$.
Since $F[A_2\cup B_1]$ is a forest, we have $e(F[A_2\cup B_1])\leq |A_2|+|B_1|-1$.
On the other hand, since $\delta\geq 2$,
we can see that $e(F[A_2\cup B_1])\geq 2|A_2|-1$ if there exists a vertex in $A_2$ incident to $e$,
and $e(F[A_2\cup B_1])\geq 2|A_2|$ if there exists no vertex in $A_2$ incident to $e$.
In both situations,
   $$2|A_2|-1\leq e(F[A_2\cup B_1])\leq |A_2|+|B_1|-1,$$
which yields that $|A_2|\leq |B_1|$.
Combining $A_1\subseteq Y_1\setminus B_1$, we obtain
  $$q+1=|A|=|A_1|+|A_2|\leq |Y_1\setminus B_1|+|B_1|=|Y_1|=q,$$
a contradiction.
Hence, $S_{n,q}^1$ is $F$-free.
It follows that $\rho(G^{\star})\geq \rho(S_{n,q}^1)$.

Now we prove that $e(G^{\star}[R])\geq 1$.
Otherwise, $e(G^{\star}[R])=0$, which implies that $G^{\star}$ is a proper subgraph of $S_{n,q}^1$.
Then, $\rho(G^{\star})<\rho(S_{n,q}^1)$, contradicting $\rho(G^{\star})\geq \rho(S_{n,q}^1)$.
Hence, $e(G^{\star}[R])\geq 1$.
This completes the proof.
\end{proof}

Combining Lemmas \ref{lem3.4} and \ref{lem3.5}, we can directly get Theorem~\ref{theorem1.4}.
Having Lemmas \ref{lemma3.1.}-\ref{lem3.5},
we are ready to complete the proofs of the remaining theorems.

\begin{proof}[\textbf{Proof of Theorem~\ref{theorem1.2}}]
(i) Recall that $G^{\star}$ is an extremal graph to ${\rm spex}_{\mathcal{P}}(n,F)$ and $G^{\star}=G^{\star}[L] \nabla G^{\star}[R]$.
Suppose that ${\rm SPEX}(n,F)=\{S_{n,(\ell-2)/2}^0\}$.
Then, $G^{\star}[L]=K_{(\ell-2)/2}$ and $e(G^{\star}[R])=0$.
Since $e(G^{\star}[R])=0$, we have $\delta=1$  by Lemma \ref{lem3.5}.
Since $G^{\star}[L]=K_{(\ell-2)/2}$, we have $\beta(F)\geq {\ell/2}$  by Lemma \ref{lem3.3}.
Combining \eqref{alg002} gives $\beta(F)=\ell/2$.

Conversely, if $\beta(F)=\ell/2$, then $\beta(F)=|A|=|B|$ by \eqref{alg002}.
We can further find that $\delta=1$ as $F$ is a tree, and ${\rm EX}(q,\mathcal{A})=\{K_{(\ell-2)/2}\}$.
By Lemma \ref{lem3.4}, $G^{\star}[L]\cong K_{(\ell-2)/2}$ and $e(G^{\star}[R])=0$, that is, ${\rm SPEX}(n,F)=\{S_{n,(\ell-2)/2}^0\}$, as desired.

(ii) Suppose ${\rm SPEX}(n,F)=\{S_{n,{(\ell-3)/2}}^{1}\}$, that is, $G^{\star}[L]\cong K_{(\ell-3)/2}$ and $e(G^{\star}[R])=1$.
Since $e(G^{\star}[R])=1$, we have $\delta\geq 2$ by Lemma \ref{lem3.5}.
Since $G^{\star}[L]\cong K_{(\ell-3)/2}$, we have $\beta(F)\geq {(\ell-1)/2}$ by Lemma \ref{lem3.3}.
Combining this with \eqref{alg002} gives that $\beta(F)=q+1={(\ell-1)/2}$, as desired.

Conversely, suppose $\beta(F)={(\ell-1)/2}$ and $\delta\geq 2$.
Combining \eqref{alg002} gives $|A|=q+1=(\ell-1)/2$.
We first claim that $G^{\star}[R]$ is $2K_2$-free.
Otherwise, $G^{\star}$ contains a copy of $K_{q,n-q}^{2}$.
Let $v_1,v_2$ be two endpoints of a longest path $P$ in $F$.
Since $F$ is not a star, the path $P$ is of length at least 3,
which implies that $v_1,v_2$ have no common neighbors.
Since $\delta\geq 2$, we have $v_1,v_2\in B$.
Set $A'=B\setminus \{v_1,v_2\}$ and $B'=A\cup \{v_1,v_2\}$.
Then $A'$ is an independent set of $F$ with $|A'|={(\ell-3)/2}=q$, and $F[B']$ consists of two independent edges and some isolated vertices.
This indicates that $F\subseteq K_{q,\ell-q}^{2}$.
However, $G^{\star}$ contains a copy of $K_{q,n-q}^{2}$,
and so contains a copy of $F$, a contradiction.
Hence, $G^{\star}[R]$ is $2K_2$-free.

We then claim that $G^{\star}[R]$ is $P_3$-free.
Since $\delta\geq 2$, we have $$\ell-1=e(F)=\sum_{v\in A}d_A(v)\geq \delta \frac{\ell-1}{2}\geq \ell-1.$$
This indicates that all vertices in $A$ are of degree 2.
Choose an arbitrary vertex $v_0\in A$.
Set $A''=A\setminus \{v_0\}$ and $B''=B\cup \{v_0\}$.
Then $A''$ is an independent set of $F$ with $|A''|={(\ell-3)/2}$, and $F[B'']$ consists of a path of length 2 with center $v_0$ and some isolated vertices.
This implies that $G^{\star}[R]$ is $P_3$-free.

Combining the above two claims, we can see that $e(G^{\star}[R])\leq 1$, and hence $G^{\star}\subseteq S_{n,{(\ell-3)/2}}^{1}$ as $q=(\ell-3)/2$.
By $\delta\geq 2$ and Lemma \ref{lem3.5}, $S_{n,{(\ell-3)/2}}^{1}$ is $F$-free.
Therefore, $G^{\star}\cong S_{n,{(\ell-3)/2}}^{1}$.
The result follows.
%
\end{proof}

\begin{proof}[\textbf{Proof of Theorem~\ref{cor4.1}}]

For non-negative integers $a,b,c$ with $a\geq b+1$ and $c\geq 1$, let $S(a,b,c)$ be the spider with $a-b-1$ legs of length 1, $b$ legs of length 2 and one leg of length $c$.
Clearly, $$|V(S(a,b,c))|=(a-b-1)+2b+c+1=a+b+c.$$
We can find integers $\alpha$ and $\gamma$ such that $0\leq \gamma \leq 1$ and $\ell-d-1=2\alpha+\gamma$.
Then $S(\alpha+\gamma+2,\alpha+1,d-2)$ is a spider of order $\ell$ and diameter $d$.

(i) Suppose first that $\ell$ is even.
Whether $d$ is even or not,
we always obtain that $\beta(S(\alpha+\gamma+2,\alpha+1,d-2))=\ell/2$.
By Theorem \ref{theorem1.2} (i), ${\rm SPEX}(n,S(\alpha+\gamma+2,\alpha+1,d-2))=\{G_{n,\ell}\}$.
Suppose now that $\ell$ is odd and $d$ is even.
It is not hard to check that $\gamma=0$, $\delta=2$ and $\beta(S(\alpha+\gamma+2,\alpha+1,d-2))=(\ell-1)/2$.
By Theorem \ref{theorem1.2} (ii), ${\rm SPEX}(n,S(\alpha+\gamma+2,\alpha+1,d-2))=\{G_{n,\ell}\}$, as desired.

(ii) Suppose that both $\ell$ and $d$ are odd.
Let $F$  be a graph of order $\ell$ and diameter $d$.
Then two endpoints of a longest path in $F$ belong to different partite sets, which implies that $\delta=1$.
On the one hand, $\beta(S(\alpha+\gamma+2,\alpha+1,d-2))=q+1=(\ell-1)/2$.
By $\delta=1$ and Theorem \ref{theorem1.2}, ${\rm SPEX}(n,S(\alpha+\gamma+2,\alpha+1,d-2))=\{S_{n,(\ell-3)/2}^{0}\}$.
This means that $S_{n,(\ell-3)/2}^{0}$ does not contain a copy of $S(\alpha+\gamma+2,\alpha+1,d-2)$.
On the other hand,
By $\delta=1$ and Lemma \ref{lem3.4}, $e(G^{\star}[R])=0$. Then,
any graph in ${\rm SPEX}(n,F)$ is a subgraph of $S_{n,q}^{0}$,
and consequently, it is also a subgraph of $S_{n,(\ell-3)/2}^{0}$ as $q+1=|A|\leq (\ell-1)/2$.
This means that ${\rm spex}(n,F)\leq \rho(S_{n,(\ell-3)/2}^{0})$,
with equality if and only if $G^{\star}=S_{n,(\ell-3)/2}^{0}$.
Therefore, if $\rho(G)\geq \rho(S_{n,(\ell-3)/2}^{0})$, then $G$ contains all trees of order $\ell$ and diameter $d$ unless $G=S_{n,(\ell-3)/2}^{0}$, as desired.
\end{proof}

\begin{proof}[\textbf{Proof of Theorem~\ref{theorem1.3}}]
We first consider the lower bound.
From \cite{Nikiforov2010} we know $\rho(S_{n,q}^0)=\frac{q-1}{2}+\sqrt{qn-\frac{3q^2+2q+1}{4}}$.
This, together with Lemma \ref{lem3.5}, gives that
   $$\rho(G^{\star})\geq \rho(S_{n,q}^1)>\rho(S_{n,q}^0)=\frac{q-1}{2}+\sqrt{qn-\frac{3q^2+2q+1}{4}}.$$

It remains the upper bound.
We shall prove that $\Delta\leq\delta-1$, where $\Delta$ is the maximum degree of $G^{\star}[R]$.
Suppose to the contrary that there exists a vertex $\widetilde{u}\in R$ with $d_{R}(\widetilde{u})\geq \delta$.
Choose a vertex $u_0\in A$ with $d_F(u_0)=\delta$.
Then we can embed $F$ into $G^{\star}$ by embedding $A\setminus\{u_0\}$ into $L$, and embedding $B\cup\{u_0\}$ into $R$ such that $\widetilde{u}=u_0$.
This contradicts that $G^{\star}$ is $F$-free.
The claim holds.
Applying $d=q-1$, $n_0=q$ and $d'=\Delta$ with Lemma \ref{lemma3.1.}, we have $\rho^{\star}\leq \rho(J^{\star})$.
By direct computation, we have
$$\rho(J^{\star})=\frac{q+\Delta-1}{2}
    +\frac12\sqrt{(q+\Delta-1)^2-4((q-1)\Delta-q (n-q))},$$
and  $$\rho(J)=\frac{q+\delta-2}{2}
    +\frac12\sqrt{(q+\delta-2)^2-4((q-1)(\delta-1)-q (n-q))}.$$
Since $n$ is sufficiently large and $\Delta\leq\delta-1$,
we obtain that $$\rho^{\star}\leq\rho(J^{\star})\leq \rho(J)=\sqrt{qn}+\frac{q+\delta-2}{2}+O(\frac{1}{\sqrt{n}}).$$
This completes the proof.
\end{proof}

\begin{proof}[\textbf{Proof of Theorem~\ref{theorem1.5}}]
Let $v^{\star}$ be the center of the spider $F$,
and let $C$ denote the set of vertices at odd distance from $v^{\star}$ in $F$.
Then $C\in \{A,B\}$.
Combining Lemma \ref{lemma3.2.}, we can observe that
\begin{eqnarray}\label{alg017}
\delta=\left\{
                                       \begin{array}{ll}
                                         1  & \hbox{if $r\geq 1$ and $s\geq 1$,} \\
                                         2  & \hbox{otherwise,}
                                       \end{array}
                                     \right.
\end{eqnarray}
and
\begin{eqnarray}\label{alg018}
\beta(F)=\nu(F)=|A|=\left\{
                                       \begin{array}{ll}
                                         (\ell-r+1)/2  & \hbox{if $r\geq 1$,} \\
                                        (\ell-1)/2  & \hbox{if $r=0$.}
                                       \end{array}
                                     \right.
\end{eqnarray}
We first give the following claim.
\begin{claim}\label{claim1.2}
$G^{\star}[R]$ is $P_3$-free.
\end{claim}

\begin{proof}
Since $F$ is not a star, we can select a leg of length $k\geq 2$, say $v^{\star}v_1\cdots v_{k}$.
Clearly, $v_{i}\in A$ for some $i\in \{1,2\}$.
Set $A'=A\setminus\{v_{i}\}$ and $B'=B\cup \{v_{i}\}$.
Then, $A'$ is an independent set of $F$ with $|A'|=|A|-1=|L|$,
and $F[B']$ consists of a path of length 2 with center $v_i$ and some isolated vertices.
Thus, $G^{\star}[R]$ is $P_3$-free.
\end{proof}

Now we distinguish two cases to complete the proof.

\noindent{{\bf{Case 1.}}} $s\geq 1$.

Suppose first that $r\geq 1$.
By \eqref{alg017} and \eqref{alg018}, we have $\delta=1$ and $\beta(F)=q+1=(\ell-r+1)/2$.
Combining Theorem \ref{theorem1.4}, we have ${\rm SPEX}(n,F)=\{S_{n,(\ell-r-1)/2}^{0}\}$, as desired.
Suppose then that $r=0$.
By \eqref{alg018}, $|L|=q=(\ell-3)/2$.
Since $s=d_F(v^{\star})\geq 2$, we can select two legs of even length, say $v^{\star}v_1v_2\cdots v_{2k_1-1}v_{2k_1}$ and $v^{\star}w_1w_2\cdots w_{2k_2-1}w_{2k_2}$.
Obviously, $v_{2k_1},w_{2k_2}\in B$.
Set $A'=B\setminus\{v_{2k_1},w_{2k_2}\}$ and $B'=A\cup \{v_{2k_1},w_{2k_2}\}$.
Then, $A'$ is an independent set of $F$ with $|A'|=|B|-2=(\ell-3)/2$,
and $F[B']$ consists of two independent edges and some isolated vertices.
This implies that $G^{\star}[R]$ is $2K_2$-free.
Combining Claim \ref{claim1.2}, we have $e(G^{\star}[R])\leq 1$, and thus $G^{\star}\subseteq S_{n,(\ell-3)/2}^1$.
On the other hand, from \eqref{alg017} we get $\delta= 2$.
By Lemma \ref{lem3.5}, $S_{n,(\ell-3)/2}^1$ is $F$-free.
Thus, $G^{\star}=S_{n,(\ell-3)/2}^1$, as desired.

\noindent{{\bf{Case 2.}}} $s=0$.

Obviously, $r\geq 2$.
Since $F$ is not a star, we have $r_1+r_2\geq 1$.
Now, we divide the proof into the following three subcases.

\noindent{{\bf{Subcase 2.1.}}} $r_1\geq 1$.

Then, there exists a leg of length $2k+1\geq 5$, say $v^{\star}v_1\dots v_{2k+1}$.
Clearly, $v^{\star},v_2,\dots,v_{2k}\in A$.
Set $A'=(A\setminus\{v_2,v_4\})\cup \{v_3\}$ and $B'=V(F)\setminus A'$.
Then, $A'$ is an independent set of $F$ with $|A'|=|A|-1=(\ell-r-1)/2$,
and $F[B']$ consists of two independent edges $v_1v_2, v_4v_5$ and some isolated vertices.
This indicates that $G^{\star}[R]$ is $2K_2$-free.
Combining Claim \ref{claim1.2}, we have $e(G^{\star}[R])\leq 1$, and thus $G^{\star}\subseteq S_{n,(\ell-r-1)/2}^1$.
On the other hand, since $\delta=2$ by \eqref{alg017},
by Lemma \ref{lem3.5}, $S_{n,(\ell-r-1)/2}^{1}$ is $F$-free.
Thus, $G^{\star}=S_{n,(\ell-r-1)/2}^1$, as desired.

\noindent{{\bf{Subcase 2.2.}}}  $r_1=0$, $r_2\geq 1$ and $r_3\in \{0,1\}$.

By \eqref{alg018}, $|A|=(\ell-r+1)/2$ and $|B|=(\ell+r-1)/2$.
Moreover, there are exactly $r$ leaves in $F$, say $v_1,v_2,\dots,v_{r}$,
which contains no common neighbors as $r_3\in \{0,1\}$.
Obviously, $v_1,v_2,\dots,v_{r}\in B$.
Set $A'=B\setminus \{v_1,v_2,\dots,v_{r}\}$ and $B'=A\cup \{v_1,v_2,\dots,v_{r}\}$.
Then $A'$ is an independent set of $F$, and $F[B']$ consists of $r$ independent edges and some isolated vertices.
Since $|L|=|A'|$, we can observe that $G^{\star}[R]$ is  $rK_2$-free,
and $S_{n,(\ell-r-1)/2}^{r}$ contains a copy of $F$.
Combining Claim \ref{claim1.2}, we can see that $G^{\star}[R]$ consists of at most $r-1$ independent edges and some isolated vertices, and thus  $G^{\star}\subseteq S_{n,(\ell-r-1)/2}^{r-1}$.

Note that $S_{n,(\ell-r-1)/2}^{r}$ contains a copy of $F$.
Then $r'\leq r$,
where $r'$ is the minimum integer such that $S_{n,(\ell-r-1)/2}^{r'}$ contains a copy of $F$.
By \eqref{alg017}, $\delta\geq 2$.
Then, from Lemma \ref{lem3.5} we know that $S_{n,(\ell-r-1)/2}^{1}$ is $F$-free, which implies that $r'\geq 2$.
Now we shall prove $r'=r$.
Otherwise, $r'<r$.
Embed $F$ into $S_{n,(\ell-r-1)/2}^{r'}$,
where $Y_1$ is the set of dominating vertices of $S_{n,(\ell-r-1)/2}^{r'}$ and $Y_2=V(S_{n,(\ell-r-1)/2}^{r'})\setminus Y_1$.
Set $V(F)\cap V(Y_1)=A'$ and $V(F)\cap V(Y_2)=B'$.
By the definition of $r'$,  $F[B']$ contains exactly $r'$ independent edges, say $e_1,e_2,\dots,e_{r'}$, and some isolated vertices.
Contracting $e_i$ as a vertex for each $i\in \{1,\dots,r'\}$ in $F$ and $S_{n,(\ell-r-1)/2}^{r'}$, we obtain a corresponding spider $F'$ and a corresponding graph $S_{n-r',(\ell-r-1)/2}^{0}$.
Then, $F'\subseteq S_{n-r',(\ell-r-1)/2}^{0}$ as $F\subseteq S_{n,(\ell-r-1)/2}^{r'}$.
Now we shall prove that $S_{n-r',(\ell-r-1)/2}^{0}$ is $F'$-free, which gives a contradiction.
By Claim \ref{claim1.2}, any leg of $F$ has at most one of these independent edges.
If $r_3=0$, then $F'$ has exactly $r-r'\geq 1$ legs of length 3 and $r'$ legs of length 2.
If $r_3=1$, then either $F'$ has exactly $r-1-r'$ legs of length 3, $r'$ legs of length 2 and one leg of length 1, or $F'$ has exactly $r-r'\geq 1$ legs of length 3 and $r'-1$ legs of length 2.
Let $A'$ and $B'$ be two partite sets of $F'$ with $|A'|\leq |B'|$.
In all situations, we can observe that $|V(F')|=\ell-r'$ and $F'$ has exactly $r-r'\geq 1$ legs of odd length.
By a similar discussion of \eqref{alg018}, we have
 $$\beta(F')=\nu(F')=|A'|=((\ell-r')-(r-r')+1)/2=(\ell-r+1)/2.$$
By the definition of $\mathcal{A}$, ${\rm EX}(q,\mathcal{A})\cong \{K_{(\ell-r-1)/2}\}$.
Then, by Lemma \ref{lem3.3}, we obtain that $S_{n-r',(\ell-r-1)/2}^{0}$ is $F'$-free, a contradiction.
Hence, $r'=r$. By the definition of $r'$, $S_{n,(\ell-r-1)/2}^{r-1}$ is $F$-free.
Recall that $G^{\star}\subseteq S_{n,(\ell-r-1)/2}^{r-1}$.
Then by the definition of $G^{\star}$, we have $G^{\star}=S_{n,(\ell-r-1)/2}^{r-1}$, as desired.

\noindent{{\bf{Subcase 2.3.}}}  $r_1=0$, $r_2\geq 1$ and $r_3\geq 2$.

Clearly, $\ell=1+3r_2+r_3=1+r+2r_2$, and consequently $r_2=(\ell-r-1)/2$.
By Claim \ref{claim1.2}, $G^{\star}[R]$ is $P_3$-free.
By \eqref{alg018}, $q=(\ell-r-1)/2$.
Combining these with Theorem \ref{theorem1.1},
$G^{\star}\subseteq S_{n,(\ell-r-1)/2}^{\lfloor(2n-\ell+r+1)/{4}\rfloor}$.
It suffices to show that $S_{n,(\ell-r-1)/2}^{\lfloor(2n-\ell+r+1)/{4}\rfloor}$ is $F$-free.
Suppose to the contrary that $S_{n,(\ell-r-1)/2}^{\lfloor(2n-\ell+r+1)/{4}\rfloor}$ contains a copy of $F$.
Then embed $F$ into $S_{n,(\ell-r-1)/2}^{\lfloor(2n-\ell+r+1)/{4}\rfloor}$,
where $Y_1$ is the set of dominating vertices of $S_{n,(\ell-r-1)/2}^{\lfloor(2n-\ell+r+1)/{4}\rfloor}$ and $Y_2=V(S_{n,(\ell-r-1)/2}^{\lfloor(2n-\ell+r+1)/{4}\rfloor})\setminus Y_1$.
Set $V(F)\cap Y_1=A'$ and $V(F)\cap Y_2=B'$.
Clearly, $F-\{v^{\star}\}$ consists of $r_2$ paths of length 2, say $P^1,P^2,\dots,P^{r_2}$,
 and $r_3$ isolated vertices, say $w_1,w_2,\dots,w_{r_3}$.
Since $S_{n,(\ell-r-1)/2}^{\lfloor(2n-\ell+r+1)/{4}\rfloor}[Y_2]$ is $P_3$-free,
 at least one vertex of $P^{i}$ belongs to $A'$ for each $i\in \{1,2,\dots,r_2\}$,
and at least one vertex of $\{v^{\star},w_1,w_2\}$ belongs to $A'$.
It follows that $|A'|\geq r_2+1=(\ell-r+1)/2$, which contradicts that $|A'|\leq |Y_1|= (\ell-r-1)/2$.
Hence, $S_{n,(\ell-r-1)/2}^{\lfloor(2n-\ell+r+1)/{4}\rfloor}$ is $F$-free.

This completes the proof of Theorem \ref{theorem1.5}.
\end{proof}

\end{document}